\documentclass[11pt,leqno]{article}
\usepackage{amssymb, amscd, amsmath, amsthm}

\allowdisplaybreaks

\def\ve{\varepsilon}
\def\intl{\int\limits}
\def\mod{\,\text{\rm mod}\;}
\def\beq{\begin{equation}}
\def\eeq{\end{equation}}
\def\cite#1{{\rm [#1]}}
\def\ol{\overline }
\def\ul{\underline }
\def\wt{\widetilde }

\def\Re{\text{\rm Re}\,}
\newtheorem{theorem}{Theorem}
\newtheorem{lemma}{Lemma}

\theoremstyle{definition}
\newtheorem{defn}{Definition}
\DeclareMathOperator{\Res}{Res}

\begin{document}

\numberwithin{equation}{section}
\title{Are there arbitrarily long arithmetic progressions in the sequence of twin primes?}

\author{J\'anos Pintz\thanks{Supported by OTKA Grants K72731, K67676 and ERC-AdG.228005.}}

\date{}
\maketitle
\section{Introduction}
\label{sec:1}

The problem in the title seemed to be out of reach of any methods before 2004.
We have still no answer for it, and it is no surprise that we will not answer it in the present work either.
However, in the last few years the following developments have been established in connection with the above problem.

\renewcommand{\thetheorem}{\Alph{theorem}}

\begin{theorem}[Green and Tao \cite{GT}]
\label{th:A}
The primes contain arbitrarily long arithmetic progressions.
\end{theorem}

\begin{theorem}[Goldston, Pintz, Y{\i}ld{\i}r{\i}m \cite{GPY1}]
\label{th:B}
If the level $\vartheta$ of distribution of primes exceeds $1/2$, then there exists a positive $d \leq C_1(\vartheta)$, such that there are infinitely many generalized twin prime pairs $n, n + d$.
If $\vartheta > 0.971$ we have $C_1(\vartheta) = 16$.
\end{theorem}

Let us call a $k$-tuple $\mathcal H = \{h_i\}^k_{i = 1}$ consisting of non-negative integers \emph{admissible} if it does not cover all residue classes modulo any prime~$p$.
Theorem~\ref{th:B} was the consequence of the sharper result that if $\vartheta > 1/2$ and $k \geq C_2(\vartheta)$, then any admissible $k$-tuple $\mathcal H$, that is the set $n + \mathcal H$ contains at least two primes for infinitely many values of~$n$.

We say that $\vartheta$ is a level of distribution of primes if for every $A > 0$ and $\ve > 0$ we have
\beq
\sum_{q \leq N^{\vartheta - \ve}} \max_{\substack {a\\ (a,q) = 1}} \biggl| \sum_{\substack{p \leq N\\ p \equiv a(q)}} \log p - \frac{N}{\varphi(q)} \biggr| \, {\ll}_{\ve, A}\, \frac{N}{(\log N)^A}.
\label{eq:1.1}
\eeq

The information that $\vartheta = 1/2$ is an admissible level of distribution of primes, the celebrated Bombieri--Vinogradov theorem, just missed the unconditional proof of the existence of infinitely many generalized twin prime pairs.
However, their theorem was crucial in the proof of
\beq
\Delta = \liminf_{n \to \infty} \frac{p_{n + 1} - p_n}{\log p_n} = 0 \qquad \text{\cite{GPY1}},
\label{eq:1.2}
\eeq
and in its improvement $(\log_2 x = \log\log x)$
\beq
\liminf_{n \to \infty} \frac{p_{n + 1} - p_n }{\sqrt{\log p_n} (\log_2 p_n)^2} < \infty \quad \text{\cite{GPY2}}.
\label{eq:1.3}
\eeq

Assuming the Elliott--Halberstam conjecture \cite{EH} to be abbreviated later by EH, which states that $\vartheta = 1$ is an admissible level, or, even that $\vartheta > 0.971$ we obtained in \cite{GPY1} infinitely many gaps of size at most $16$ in the sequence of primes, that is $C_1(0.971) = 16$.
In the following let $p'$ denote the prime following~$p$.

The aim of the present work is to combine the methods of \cite{GT} and \cite{GPY1} in order to show, even in a stronger form with consecutive primes $p, p' = p + d$, the following result.

\renewcommand{\thetheorem}{\arabic{theorem}}
\setcounter{theorem}{0}

\begin{theorem}
\label{th:1}
If the level $\vartheta$ of distribution of primes exceeds $1/2$, then there exists a positive $d \leq C_1(\vartheta)$ so that there are arbitrarily long arithmetic progressions of primes~$p$ such that $p' = p + d$ is the next prime for each element of the progression.
If $\vartheta > 0.971$ then the above holds for some $d$ with $d \leq 16$.
\end{theorem}

In such a way we can show a positive answer to a weaker form of the question mentioned in the title (where twin primes are substituted by generalized twin primes) under the unproved condition that the exponent $1/2$ in the Bombieri--Vinogradov theorem can be improved to a $\vartheta =  1/2 + \delta$, where $\delta$ is an arbitrarily small, but fixed positive number.
We have to note, however, that such a quantitatively tiny improvement is probably very difficult.
For example, the Generalized Riemann Hypothesis (GRH) trivially implies that $\vartheta = 1/2$ is admissible but it does not provide any admissible level beyond that.
Our analysis, more exactly, Theorem~\ref{th:5} will show that the answer for the question in the title would be positive if one could show $\pi_2(x) \geq cx/\log^2 x$ for the number of twin primes up to~$x$.

\section{The methods of Green--Tao and Goldston--Pintz--Y{\i}ld{\i}r{\i}m}
\label{sec:2}

At the first sight it seems that there is no serious problem in combining the methods of \cite{GT} and \cite{GPY1} in order to show Theorem~\ref{th:A}, since \cite{GT} uses weights from Selberg's sieve
\beq
\Lambda_R(n) := \sum_{d \leq R,\ d \mid n} \mu(d) \log \frac{R}{d},
\label{eq:2.1}
\eeq
and applies the method of Goldston and Y{\i}ld{\i}r{\i}m to construct a pseudo-random measure where the set of primes has positive density.
Similar weights are used in the proof of Theorem~B in \cite{GPY1}.
There are, however, very important differences as well.

1) The weights \eqref{eq:2.1} and the method used by Goldston and Y{\i}ld{\i}r{\i}m \cite{GY} led just to the weaker result $\Delta = 1/4$.
The value of the parameter $R$ was a very small power of the size of the primes~$p \asymp N$.
In this approach of Goldston and Y{\i}ld{\i}r{\i}m primes were searched in admissible tuples
\beq
\mathcal H = \{h_1, \dots, h_k \}, \ \ 0 \leq h_1 < h_2 < \dots < h_k, \ \ h_i \in \mathbb Z
\label{eq:2.2}
\eeq
and they used the weights
\beq
\Lambda_R(n; \mathcal H) := \prod^k_{i = 1} \Lambda_R(n + h_i).
\label{eq:2.3}
\eeq

2) In the work \cite{GPY1} yielding Theorem~\ref{th:B} and \eqref{eq:1.2}, it was crucial to use a maximal possible value $R$ with $R^2 \leq N^{\vartheta - \ve}$ (that is, $R = N^{\frac{\vartheta - \ve}2}$) allowed by the known information about primes.
Additionally, the weights \eqref{eq:2.3} had to be replaced by (the factor $1/(k + \ell)!$ is insignificant here, it serves just for normalization)
\beq
\Lambda_R(n; \mathcal H, k + l) := \frac1{(k\! +\! l)!} \sum_{d \mid P_{\mathcal H}(n)} \mu(d) \log^{k + l} \frac{R}{d}, \ \ P_{\mathcal H}(n) := \prod\limits^k_{i = 1} (n + h_i)
\label{eq:2.4}
\eeq
with $l$ and $k$ being bounded but arbitrarily large, $l = o(k)$.

If one imposed any upper bound on at least one of $k$ or $l$, our method would not lead to $\Delta = 0$ in \eqref{eq:1.2}.
Similarly if we had at our disposal just any level $\vartheta < 1/2$ (say $\vartheta = 0.49999$) we would not be able to deduce $\Delta = 0$ (it would lead however to $\Delta \leq c_0(\vartheta)$, with $c_0 (\vartheta) \to 0$ as $\vartheta \to 1/2$).
In case of Theorem~\ref{th:B}, if one uses $k$ and $l$ with $\min(l,k) \leq C_0$, then beyond $\vartheta > 1/2$ we would need $\vartheta > 1/2 + c(C_0)$.

3) The main difficulty is that the number of generalized twin primes up to~$X$, produced by Theorem~\ref{th:B} did not yield a lower bound for them beyond the very weak implicit bound
\beq
N^{1 - C(k)/\log\log N}.
\label{eq:2.5}
\eeq
This is far from the expected order of magnitude
\beq
\frac{\mathfrak S(d)N}{\log^2 N} \sim \frac{\mathfrak S(d)\pi(N)}{\log N},
\label{eq:2.6}
\eeq
which is still just a density of size $c/\log N$ among the primes.
We also do not know whether the primes $p_n$ satisfying $p_{n + 1} - p_n < \eta \log p_n$ have a positive density if $\eta < 1/4$.
It will turn out, however, that the first part of the proof (Section~\ref{sec:6}) can help to answer this question positively too.
We shall return to this problem in a later work.

\section{The principal idea of the proof}
\label{sec:3}

The seemingly simplest solution, to work with the same weight function during the whole proof, seems to be impossible, due to the reasons mentioned in Section~\ref{sec:2}.
So we will use two different weight functions:

(i) the first one of type \eqref{eq:2.4} to produce the generalized twin prime pairs (with an $R = N^{\vartheta/2 - \ve})$, afterwards

(ii) the second one of type \eqref{eq:2.3} with a different $R = N^\delta$, with a small value of~$\delta$.

However, this still does not solve our problem that we are not able to produce a good lower bound for the number of generalized twin primes.
How do we find the set (or, more precisely the pseudorandom measure~$\nu$) where the generalized twin primes produced in the first step are contained with positive density?
In case of the primes in Green--Tao's theorem this set was the set of ``almost primes'', more precisely the measure~\eqref{eq:2.1}.

It is of crucial importance that the generalized twin prime pairs $n, n + d$ $(n \in N_1)$ were produced in \cite{GPY1} as two primes within a set of type
\beq
\bigl\{n + h_i\bigr\}^k_{i = 1} := n + \mathcal H, \ \ \mathcal H \ \text{ admissible.}
\label{eq:3.2}
\eeq
The main idea is to embed these twin primes into the set of almost prime $k$-tuples with pattern~$\mathcal H$; more precisely to consider just that part $N_2 \subset N_1$ for which the other components $n + h_j$ are almost primes (we could simply say that all components are almost primes since primes are considered themselves as almost primes).
In the usual literature almost primes $P_r$ are considered as numbers with a bounded number of prime factors.
(A $P_r$ number is by definition a number with at most $r$ prime factors.)
The sieve methods producing almost primes usually automatically produce almost primes satisfying
\beq
P^-(n) > n^c, \qquad c > 0, \ \text{ fixed},
\label{eq:3.3}
\eeq
where $P^-(n)$ denotes the least prime factor of~$n$.
On the other hand, numbers with property \eqref{eq:3.3} are automatically $P_r$ almost primes for any $r \geq \lfloor 1/c\rfloor $.
It will be important in the sequel that talking about almost primes we mean always the stronger sense \eqref{eq:3.3}.
Concerning almost primes we have to mention an analogue of the Dickson--Hardy--Littlewood prime $k$-tuple conjecture for almost primes, proved first by Halberstam and Richert \cite{HR} (cf.\ Theorem 7.4)
later in a sharper form by Heath-Brown \cite{Hea}.
Let $\Omega(n)$ denote the number of prime divisors of~$n$.

\renewcommand{\thetheorem}{\Alph{theorem}}
\setcounter{theorem}{2}

\begin{theorem}
\label{th:C}
For any admissible $k$-tuple $\mathcal H$ we have infinitely many values of~$n$ such that $P_{\mathcal H}(n) = \prod\limits^k_{i = 1} (n + h_i)$ is a $P_K$ almost prime with $K \leq C_3(k)$, that is, $\Omega(P_{\mathcal H}(n)) \leq C_3(k)$.
\end{theorem}

Halberstam and Richert proved additionally

(i) $P^-(n + h_i) \geq n^{c/k} \quad (i = 1,2,\dots, k), \quad
C_3(k) = (1 + o(1)) k \log k$,

\noindent
where $c = 2/5$ can be chosen if $k$ is large enough,
while Heath-Brown showed that there are infinitely many $n$ values with

(ii) $\max\limits_{i \leq k} \Omega (n + h_i) \leq (1 + o(1)) \dfrac2{\log 2} \log k$.

Theorem~\ref{th:C} raises the problem, whether we can combine it with Theorem~\ref{th:B} in order to show the existence of infinitely many $n$ values such that all elements $n + h_i$ are almost primes and at least two of them are primes
under the hypothesis $\vartheta > 1/2$.
It turns out that it is ``easier'' to prove a common generalization of Theorems~\ref{th:A}, \ref{th:B}, and \ref{th:C} than just of Theorems~\ref{th:A} and \ref{th:B}.
In fact our strategy in proving Theorem~\ref{th:1} will be as follows.

\smallskip
Step 1.
We show that the procedure of Theorem~\ref{th:B} yields in fact for most of the ``good'' values of~$n$ not just two primes in the $k$-tuple considered, but also almost primes for each component $n + h_i$, $i \leq k$.

\smallskip
Step 2.
We show that although the produced generalized twin primes probably do not form a set of positive density within the set of all generalized twin primes, however, the produced $n$ values with at least two primes $n + h_i$ form a set of positive density among all $n$'s satisfying $P^-(n + h_i) > n^{c_1(k)}$ for each $h_i \in \mathcal H$.

\smallskip
Step 3.
To show that the measure $\ol \nu(n)$ derived from the weight function $\Lambda_R(n; \mathcal H)$ with the new parameter $R = N^\delta$ (with a small positive constant~$\delta$) is a pseudorandom measure and the produced tuples form a set of positive measure.

In this way we obtain a common generalization of Theorems~\ref{th:A}, \ref{th:B} and \ref{th:C}, thereby a sharper form of Theorem~\ref{th:1} as follows.

\renewcommand{\thetheorem}{\arabic{theorem}}
\setcounter{theorem}{1}

\begin{theorem}
\label{th:2}
Let us suppose that $\vartheta = 1/2 + \delta > 1/2$ is a level of distribution of primes.
Let $\mathcal H = \{h_i\}^k_{i = 1}$ be an admissible $k$-tuple with $k \geq C_0(\vartheta) = \bigl(2 \lceil 1/2\delta \rceil + 1 \bigr)^2$.
Then we have with some values $b_i \leq C_3(k)$, with at least two of the $b_i$'s being equal to~$1$, arbitrarily long arithmetic progressions of $n$ values such that $\Omega(n + h_i) = b_i$, $P^-(n + h_i) > n^{c_1(k)}$.
If $\vartheta \geq 0.971$ then this is true for $k \geq 6$.
Further, the same is true with at least two consecutive primes in the set $n + \mathcal H$.
\end{theorem}

\noindent
{\bf Remark.}
In other words, for sufficiently large $k \geq C_0(\vartheta)$ the same multiplicative structure of $n + \mathcal H$ (i.e.\ $\Omega(n + h_i) = b_i$) containing at least two primes and almost primes elsewhere, appears arbitrarily many times with equal distances among the neighboring constellations (i.e.\ the $n$ values, thereby the primes $n + h_i$ and $n + h_j$ forming an arithmetic progression).
We can even require that $n + h_i$ should have the same exponent pattern, where the multiset $(\alpha_1, \dots, \alpha_j)$ is defined as the exponent pattern of $n$ (see \cite{GGPY})
\beq
n = \Pi p^{\alpha_i}_{\nu_i}, \quad p_{\nu_i} \in \mathcal P .
\label{eq:3.4}
\eeq

\smallskip
\noindent
{\bf Remark.}
We can not determine in advance the places of the two primes in the constellation, neither the values of $b_\nu$ apart from $b_\nu \leq C_3(k)$.
However, they will be the same for each translated copy, that is, for each element $n$ of the arithmetic progressions.

\smallskip
\noindent
{\bf Remark.}
If we had $b_i = 1$ $(i = 1, \dots, k)$, that is, $C_3(k) = 1$, this would be Dickson's conjecture about prime $k$-tuples, also called the Hardy--Littlewood prime $k$-tuple conjecture.

\smallskip
\noindent
{\bf Remark.}
The lower bound $k \geq \bigl(2 \lceil 1/2\delta \rceil + 1 \bigr)^2$ can be improved if $\delta$ is not too small.

\section{Further results about generalized twin primes}
\label{sec:4}

The execution of Steps 1 and 2 reveals already interesting properties of the weights \eqref{eq:2.4} and (coupled with other arguments in some cases) yields or helps to yield important consequences, such as the positive  proportion of small gaps of size at most $\eta \log p$ between consecutive primes~$p$ and $p'$ for any $\eta > 0$.
Therefore it is worth formulating the result of these steps as the following separate theorem.

\begin{theorem}
\label{th:3}
Suppose that the level of distribution of primes is $\vartheta = \frac12 + \delta > \frac12$.
If $\mathcal H = \{h_i\}^k_{i = 1}$ is any admissible $k$-tuple with $k \geq C_0(\vartheta) = \bigl(2 \lceil 1/2\delta\rceil + 1\bigr)^2$, then the number of $n \leq N$ for which $n + \mathcal H$ contains at least two consecutive primes and almost primes in each component satisfying $P^-(n + h_\nu) > n^{c_1(k)}$ is at least
\beq
c_1(k, \mathcal H) \frac{N}{\log^k N}
\label{eq:4.1}
\eeq
with some $c_1(k, \mathcal H)$, depending on~$k$ and $\mathcal H$.
Choosing the $k$-tuple with a possibly small diameter we obtain at least
\beq
c_2(k) \frac{N}{\log^k N}, \qquad k = \left(2 \left\lceil \frac1{2\delta}\right\rceil + 1 \right)^2
\label{eq:4.2}
\eeq
generalized twin prime pairs $n$, $n + d$ with a difference
\beq
d \leq C^*(\vartheta) = (1 + o(1)) k\log k = \left(2 + o(1) \right) \delta^{-2} \log \frac1{\delta}.
\label{eq:4.3}
\eeq
If the Elliott--Halberstam conjecture or at least $\vartheta > 0.971$ holds then we obtain at least
\beq
c_1 \frac{N}{\log^6 N}
\label{eq:4.4}
\eeq
generalized twin prime pairs $n$, $n + d$ with a positive $d \leq 16$.
\end{theorem}

\noindent
{\bf Remark.}
It is actually the number \eqref{eq:4.1} of the obtained $k$-tuples which enables the use of the procedure of Green and Tao (with a parameter
$R < N^{c(k, m)}$, where $m$ is the number of terms in the arithmetic progression) to assure the existence of arbitrarily long arithmetic progressions among the $k$-tuples $n + \mathcal H$ with the above properties.

\smallskip
The earliest known written formulation of the twin prime conjecture seems to be due to de Polignac \cite{Pol} from the year 1849.
This conjecture was already about prime pairs with a general even difference~$d$.
This also indicates that the original twin prime conjecture arose already earlier.
If $d = 2$ or $4$ then a pair of primes~$n$, $n + d$ must be clearly a consecutive prime-pair if $n > 3$.
On the other hand, this is probably not true for all prime pairs with $d \geq 6$.
This would follow for example from the special case $k = 3$ of the Dickson--Hardy--Littlewood prime $k$-tuple conjecture, stating that any admissible $k$-tuple contains infinitely many prime $k$-tuples.
The original de Polignac conjecture \cite{Pol} stated that for any even~$d$
\beq
n, \ n + d \ \text{ are consecutive primes for infinitely many values } n.
\label{eq:4.5}
\eeq
A weaker form of this conjecture would be that for any even~$d$
\beq
n, \ n + d \ \text{ are both primes for infinitely many values } n.
\label{eq:4.6}
\eeq

As long as in general nearly nothing was known about either \eqref{eq:4.5} or \eqref{eq:4.6} there was not much reason to discuss the difference between \eqref{eq:4.5} and \eqref{eq:4.6}.
Now, in view of our Theorem~\ref{th:B}, under the plausible hypothesis $\vartheta > 1/2$ it seems to be worth to discuss these aspects.
To formulate the results more easily we introduce the following definitions.

\begin{defn}
\label{def:1}
We will call an even number $d$ a de Polignac number in the strong sense (briefly strong de Polignac number) if \eqref{eq:4.5} is true for it.
\end{defn}

\begin{defn}
\label{def:2}
We will call an even number $d$ a de Polignac number in the weak sense (briefly weak de Polignac number) if \eqref{eq:4.6} is true for it.
\end{defn}

Let us denote the set of all strong and weak de Polignac numbers, respectively, by $\mathcal D_s$ and $\mathcal D_w$.
The fact that for $k > C_0(\vartheta)$  every admissible $k$-tuple $\mathcal H$ contains infinitely many times at least two primes implies that the set $\mathcal D_w$ of weak de Polignac numbers has a positive lower density depending on~$\vartheta$.
However, the earlier results could not show that apart from the smallest element $d_0$ of $\mathcal D_w$ any of the others would satisfy \eqref{eq:4.5} due to the possible existence of another prime between $n$ and $n + d$, if $d \in \mathcal D_w$, $d > d_0$.
In such a way we did not have any more information about $\mathcal D_s$ beyond $\mathcal D_s \neq \emptyset$.
This is still highly non-trivial (and still not known unconditionally) since without any hypothesis the following three assertions are clearly equivalent:

(i) there exists at least one strong de Polignac number $(\mathcal D_s \neq \emptyset)$,

(ii) there exists at least one weak de Polignac number $(\mathcal D_w \neq \emptyset)$,

(iii) there are infinitely many bounded gaps between consecutive primes, that is, $\liminf\limits_{n \to \infty} (p_{n + 1} - p_n) < \infty$.

Now, Theorem~\ref{th:3} changes this great difference about our present knowledge of the cardinality of the weak and strong de Polignac numbers, namely, that under the assumption $\vartheta > 1/2$
\beq
\ul d(\mathcal D_w) \geq c(\vartheta), \ \text{ whereas } \ |\mathcal D_s| \geq 1,
\label{eq:4.7}
\eeq
where $\ul d(X)$ denotes the lower density, $|X|$ the cardinality of a set~$X$.

In fact, Theorem~\ref{th:3} implies with some relatively easy elementary arguments the following

\begin{theorem}
\label{th:4}
Let us suppose that the primes have a level of distribution $\vartheta = \frac12 + \delta$.
Let $k = \bigl(2 \lceil 1/2\delta \rceil + 1 \bigr)^2$,
$P := P(k) := \prod\limits_{p \leq k} p$.
Then Theorem~\ref{th:3} is true in the stronger form that $n + \mathcal H$ contains at least two consecutive primes.
As a consequence of this we have ($\varphi$ is Euler's totient function)
\beq
\ul d(\mathcal D_w) \geq \ul d(\mathcal D_s) \geq \frac{\varphi(P)}{Pk(k - 1)} \sim \frac{e^{-\gamma}}{k^2 \log k} \sim \frac{\delta^4}{2e^\gamma \log (1/\delta)}
\label{eq:4.8}
\eeq
as $\delta \to 0$.
If EH or at least $\vartheta > 0.971$ is true then
\beq
\ul d(\mathcal D_w) \geq \ul d(\mathcal D_s) \geq \frac2{225}.
\label{eq:4.9}
\eeq
\end{theorem}

\noindent
{\bf Remark.}
The above theorem shows that under the Elliott--Halberstam conjecture at least approximately 1.78 percentage of all even numbers appear infinitely often as the difference of two consecutive primes, thereby satisfying the original de Polignac conjecture \eqref{eq:4.5}.
Furthermore, the set of such even numbers has a positive lower density for any level $\vartheta > 1/2$ of the distribution of primes.

\section{Preparation for Step~1.
Notation}
\label{sec:5}

In this section we will make preparations for the proof of Theorem~\ref{th:3}, which is a refinement of Theorem~\ref{th:B}.
Since we will follow the version of \cite{GMPY} we will reintroduce here its notation and describe the necessary changes to obtain the same result with the extra requirement that (in case of $\vartheta > 1/2$) apart from the two primes we obtain almost primes with $P^-(n + h_i) > n^{c_1(k)}$ in all components.
We emphasize here that a different notation will be used in Section~\ref{sec:10}, when we describe the changes in the procedure of \cite{GT}, since unfortunately the same variables $k, h_1, \dots, h_k$ refer to different quantities in the works \cite{GMPY} and \cite{GT}.

Let $n \sim N$ mean $N < n \leq 2N$, let $P$ denote the set of primes, and let $k, l$ be arbitrary bounded integers,
\beq
H \ll \log N \ll \log R \ll \log N, \qquad H \to \infty,
\label{eq:5.1}
\eeq
\beq
\mathcal H = \bigl\{ h_i \bigr\}^k_{i = 1} \subseteq [0, H], \quad h_i < h_{i + 1}, \ \ h_i \in \mathbb Z,
\label{eq:5.2}
\eeq
where $\mathcal H$ is an admissible $k$-tuple.
For the aim of later use in other works we allow here $H \to \infty$, whereas for the present work it would be enough to suppose $H \leq C(k)$.
Constants $c, C, \ve$ may be different at different occurrences and they might depend on $k$, $l$ and $\vartheta$, as well as the constants implied by $\ll$, and $O$ symbols (without indicating the dependence); $\log_\nu x$ denotes the $\nu$-fold iterated logarithmic function.
Differently from earlier parts of the work we use the notation $\Omega(p)$ and its multiplicative extension $\Omega(d)$ for squarefree $d$ to denote the set of those numbers $n$ for which
\beq
d \mid P_{\mathcal H}(n) := \prod^k_{i = 1} (n + h_i) \Longleftrightarrow n \in \Omega(d).
\label{eq:5.3}
\eeq
Further, in accordance with \eqref{eq:2.4} we define
\beq
\Lambda_R(n; \mathcal H, a) := \sum_{n \in \Omega(d)} \lambda_R(d; a), \qquad
\lambda_R(d; a) := \frac{\mu(d)}{a!} \left(\left(\log \frac{R}{d} \right)_+\right)^a
\label{eq:5.4}
\eeq
where $y_+ = y$ for $y \geq 0$ and $0$ otherwise.
Let $\theta(n) = \log n$ if $n \in P$ and $0$ otherwise.
The singular series is defined, as usual, by
\beq
\mathfrak{S}(\mathcal H) = \prod_p \left(1 - \frac{|\Omega(p)|}{p}\right) \left(1 - \frac1{p}\right)^{-k},
\label{eq:5.5}
\eeq
where $|X|$ denotes the cardinality of the set~$X$.
Theorem~\ref{th:B} was the immediate consequence of Lemma~\ref{lem:1} and the case $h \in \mathcal H$ of Lemma~\ref{lem:2}.

\begin{lemma}
\label{lem:1}
For a sufficiently large $C > C(k,l)$ and $R \leq \sqrt{N}/(\log N)^C$ we have
\begin{align}
\label{eq:5.6}
S_0(\mathcal H) &= \sum_{n \sim N} \Lambda_R(n; \mathcal H, k + l)^2 \\
&= \frac{\mathfrak S(\mathcal H)}{(k + 2l)!} {2l\choose l} N(\log R)^{k + 2l} + O\bigl(N(\log N)^{k + 2l - 1}(\log_2 N)^c\bigr). \notag
\end{align}
\end{lemma}

\begin{lemma}
\label{lem:2}
If the level of distribution of primes is $\vartheta$, $\ve > 0$, then for $R \leq N^{(\vartheta - \ve)/2}$ we have for $R \leq H$, $m = 1$ if $h \in \mathcal H$, $m = 0$ if $h \notin \mathcal H$
\begin{align}
\label{eq:5.7}
S_1(\mathcal H) :&= \sum_{n \sim N} \theta(n + h) \Lambda_R(n; \mathcal H, k + l)^2\\
 &= \frac{\mathfrak S(\mathcal H\cup \{h\})}{(k + 2l + m)!} {2(l + m)\choose l + m} N(\log R)^{k + 2l + m}\notag\\
 &\quad + O\bigl(N(\log N)^{k + 2l + m - 1}(\log_2 N)^c\bigr).\notag
\end{align}
\end{lemma}

The proof of Theorem~\ref{th:B} follows from these lemmas by
\begin{align}
\label{eq:5.8}
&\sum_{n \sim N} \biggl(\sum_{h \in \mathcal H} \theta(n + h) - \log 3N \biggr) \Lambda_R(n; \mathcal H, k + l)^2 \\
&= \! \frac{\mathfrak S(\mathcal H)}{(k\! +\! 2l)!} {2l\choose l} N\log N(\log R)^{k + 2l} \biggl(\!\frac{k\cdot 2(2l + 1)}{(k \! +\! 2l\! +\! 1)(l\! +\! 1)} \cdot \frac{\vartheta\! -\! \ve}{2} -\! 1 + o(1) \!\biggr)\! >\! 0\notag
\end{align}
if the constants $\vartheta, k, l$ satisfy the crucial inequality
\beq
\frac{k}{k + 2l + 1}\, \frac{2l + 1}{l + 1} \vartheta > 1,
\label{eq:5.9}
\eeq
since $\ve$ can be chosen arbitrarily small after $k$ and $l$ are chosen.

The crucial property of the weights $\Lambda_R(n; \mathcal H, k + l)^2$ to be proved is that it is concentrated so strongly for almost prime $k$-tuples satisfying $P^-(P_{\mathcal H}(n)) > R^\eta$ for any $\eta < c(k,l)$ that the sum of those weights
$\Lambda_R(n; \mathcal H, k + l)^2$ $(n \sim N)$
for which $P^-(P_{\mathcal H}(n)) < R^\eta$, is negligible compared with the total sum for all $n \sim N$ if $N \to \infty$, $\eta \to 0$ $(k, l, \mathcal H$ being fixed).
The same is true for the weighted sum of primes (cf.\ Lemmas~\ref{lem:2} and \ref{lem:5}), although this is not needed to prove Theorems \ref{th:1}--\ref{th:4}.

\section{The execution of Step 1}
\label{sec:6}

The mentioned property of the sieve weights
$\Lambda_R(n; \mathcal H, k + l)$ can be expressed by

\begin{lemma}
\label{lem:3}
Let $N^{c_0} < R \leq \sqrt{N/q}(\log N)^{-C}$, $q \in \mathcal P$, $q = R^\beta$, $\beta < c_0$, where $c_0$ and $C$ are suitably chosen constants depending on $k$ and $l$.
Then we have
\beq
\sum_{\substack{n \sim N\\ q \mid P_{\mathcal H}(n)}} \Lambda_R(n; \mathcal H, k + l)^2 \ll \frac{\beta}{q} \sum_{n \sim N} \Lambda_R(n; \mathcal H, k + l)^2.
\label{eq:6.1}
\eeq
\end{lemma}

This immediately implies

\begin{lemma}
\label{lem:4}
Let $N^{c_0} < R \leq N^{1/(2 + \eta)} (\log N)^{-C}$, $\eta > 0$.
Then we have
\beq
\sum_{\substack{n \sim N\\
(\mathcal P_{\mathcal H}(n), P(R^\eta)) > 1}} \Lambda_R(n; \mathcal H, k + l)^2 \ll \eta \sum_{n \sim N} \Lambda_R(n; \mathcal H, k + l)^2.
\label{eq:6.2}
\eeq
\end{lemma}

\noindent
{\bf Remark 1.}
In some cases we need to use an analogue of \eqref{eq:6.2} with the product of the weights $\Lambda_R(n; \mathcal H, k + l_1)$ and $\Lambda_R(n; \mathcal H, k + l_2)$ in place of $\Lambda_R(n; \mathcal H, k + l)^2$ with different values of $l_1$ and $l_2$.
However, a simple use of Cauchy's inequality reduces the estimates of these quantities to \eqref{eq:6.2}.
The same applies to Lemma~\ref{lem:6}, which is a simple consequence of Lemma~\ref{lem:4}.

\medskip

If the above sum is twisted by primes we can only prove an analogue of Lemma~\ref{lem:4}.

\begin{lemma}
\label{lem:5}
Let $N^{c_0} \leq R \leq N^{(\vartheta - \ve)/(2 + \eta)} (\log N)^{-C}$, $0 < \eta < c_0$, $\ve > 0$.
Let $h \leq H$, and $m = 1$ if $h \in \mathcal H$, $m = 0$ if $h \notin \mathcal H$. Then
\begin{align}
\label{eq:6.3}
&\sum_{\substack{n \sim N\\
(\mathcal P_{\mathcal H}(n), P(R^\eta)) > 1}} \theta(n + h) \Lambda_R(n; \mathcal H, k + l)^2 \\
&\ll \eta \sum_{n \sim N} \theta(n + h) \Lambda_R(n; \mathcal H, k + l)^2
\notag\\
&\quad + O\left(N\left((\log_2 N)^c\left((\log N)^{k + 2l + m - 1} + (\log N)^{k + l - \frac12}\right)\right)\right).\notag
\end{align}
\end{lemma}

\noindent
{\bf Remark 2.}
We will not investigate in the present work the dependence of the sign $\ll$ on $k$ and $l$, although it has some significance for some applications.

\smallskip
\noindent
{\bf Remark 3.}
In the present applications Lemma~\ref{lem:5} will not be used.
However, it has a significance in other applications and also here if the dependence on $k$ and $l$ is also considered.
At any rate the proof of Lemmas \ref{lem:3}--\ref{lem:5} will be very similar.

\begin{proof}[Proof of Lemma~\ref{lem:3}]
We will follow the proof of \cite{GMPY} and just point out the differences.
In evaluating the expression (1.4) in \cite{GMPY} we have to take into consideration the extra condition $q \mid \mathcal P_{\mathcal H}(n) \Leftrightarrow n \in \Omega(q)$ which has to be added to the conditions $n \in \Omega(d_1)$, $n \in \Omega(d_2)$.
Therefore the critical quantity $\mathcal T$ in the main term $N \mathcal T$ will take now the form $N\mathcal T'_q$
\beq
\mathcal T'_q = \sum_{d_1, d_2} \frac{|\Omega([d_1, d_2, q])|}{[d_1, d_2, q]} \lambda_R(d_1; k + l) \lambda_R(d_2; k + l).
\label{eq:6.5}
\eeq
Due to the multiplicative property of $\Omega$ this will mean that the main term will be now
\beq
\frac{N |\Omega(q)|}{q} \widetilde{\mathcal T}_q \quad \text{ with } \quad  \mathcal T'_q = \frac{|\Omega(q)|}{q} \widetilde{\mathcal T}_q,
\label{eq:6.6}
\eeq
where $\widetilde{\mathcal T}_q$ can be expressed similarly to \cite{GMPY} by the new generating function $\wt{F}_q = F^\#_q \cdot F_q$, where $\bold  s = (s_1, s_2)$ and $F_q(\bold  s)$ is up to the missing term for $p = q$ the same as $F$ in \cite{GMPY}:
\beq
\aligned
F^\#_q(\bold  s) :&= \left(1 - \frac1{q^{s_1}}\right) \left(1 - \frac1{q^{s_2}}\right), \\
F_q(\bold  s) :&= \prod_{p \neq q} \left(1 - \frac{|\Omega(p)|}{p} \left(\frac1{p^{s_1}} + \frac1{p^{s_2}} - \frac1{p^{s_1 + s_2}} \right)\right).
\endaligned
\label{eq:6.7}
\eeq

Analogously to (1.5) of \cite{GMPY} we can define now
\beq
\wt G_q(\bold  s) := \wt F_q(\bold  s) \left(\frac{\zeta(s_1 + 1) \zeta(s_2 + 1)}{\zeta(s_1 + s_2 + 1)}\right)^k := F^\#_q(\bold  s) G_q(\bold  s).
\label{eq:6.8}
\eeq

The appearance of the term $F^\#_q(\bold  s)$ in $\wt F_q(\bold  s)$ causes additional difficulties in evaluating the expression $\wt{\mathcal T}_q$ compared to that of the analogous quantity $\mathcal T$ of \cite{GMPY}.
We have, namely
\begin{align}
\label{eq:6.9}
\wt{\mathcal T}_q :&= \frac1{(2\pi i)^2} \intl_{(1)} \intl_{(1)} G_q(s_1, s_2) \left( \frac{\zeta(s_1 + s_2 + 1)}{\zeta(s_1 + 1) \zeta(s_2 + 1)} \right)^k \times \\
&\quad \times \frac{R^{s_1 + s_2} - (R/q)^{s_1} R^{s_2} - R^{s_1} (R/q)^{s_2} + (R/q)^{s_1 + s_2}}{(s_1 s_2)^{k + l + 1}} ds_1 ds_2,
\notag
\end{align}
where $\int_{(\beta)}$ means integration over the vertical line $\Re z = \beta$.

This means that due to the numerator of the last term above we need to evaluate (or at least estimate) integrals of the form
\beq
\mathcal T_q(R_1, R_2) := \frac1{(2\pi i)^2} \intl_{(1)} \intl_{(1)} Z_q(s_1, s_2) \frac{R^{s_1}_1 R^{s_2}_2}{(s_1 + s_2)^k(s_1 s_2)^{l + 1}} ds_1 ds_2,
\label{eq:6.10}
\eeq
where analogously to \cite{GMPY} we define
\beq
Z_q(\bold  s) = G_q(\bold  s) \left( \frac{(s_1 + s_2) \zeta(s_1 + s_2 + 1)}{\bigl(s_1 \zeta(s_1 + 1) s_2 \zeta(s_2 + 1)\bigr)}\right)^k;
\label{eq:6.11}
\eeq
however, in contrast to \cite{GMPY}, where $R_1 = R_2 = R$ we have here now, with $R_0 = R$ or $R/q$,
\beq
R_1 := R^a_0, \quad R_2 := R_0, \quad a:= 1 + \alpha, \quad
-2c_0 < \alpha < 2c_0.
\label{eq:6.12}
\eeq

First we will consider the changes in the contribution of the main term corresponding to the residue $s_1 = s_2 = 0$.
This term is now, by an argument similar to \cite{GMPY},
\begin{align}
\label{eq:6.13}
\mathcal T_{q,0}(R_1, R_2) = \underset{s_2 = 0}{\Res}\ \underset{s_1 = 0}{\Res}  &=
\frac1{(2\pi i)^2} \intl_{C_2} \intl_{C_1} \frac{Z_q(s_1, s_2)R^{as_1 + s_2}_0}{(s_1 + s_2)^k (s_1 s_2)^{l + 1}} ds_1 ds_2 \\
&= \frac1{(2\pi i)^2} \intl_{C_3} \intl_{C_1} \frac{Z_q(s, s\zeta)R^{s(a + \xi)}_0}{(\xi + 1)^k \xi^{l + 1} s^{k + 2l + 1}} ds d\xi
\notag
\end{align}
where we wrote $s_1 = s$, $s_2 = s\xi$, and $C_1, C_2, C_3$ are circles with $|s_1| = |s| = \varrho$,
$|s_2| = \varrho/2$, $|\xi| = 1/2$, respectively, with a small $\varrho > 0$.
Using the analogue of (1.6) of \cite{GMPY} for $G_q(\bold  s)$ we can write $\mathcal T_{q,0}$ as
\beq
\mathcal T_{q,0} (R_1, R_2) = \frac{Z_q(0,0)}{(2\pi i)(k + 2l)!}(\log R_0)^{k + 2l} \mathcal T_{q,1} (a) + O \bigl((\log N)^{k + 2l - 1} (\log_2 N)^c\bigr),
\label{eq:6.14}
\eeq
since $\frac{\partial^j}{\partial s^j} Z_q(0,0) \ll (\log_2 N)^c$ if $j \leq C(k,l)$.
The main term is
\begin{align}
\label{eq:6.15}
\mathcal T_{q,1}(a) :&= \frac1{2\pi i} \intl_{C_3} \frac{(a + \xi)^{k + 2l}}{(\xi + 1)^k \xi^{l + 1}} d\xi \\
&= \frac1{l!} \left[ \left( \frac{d}{d\xi}\right)^l \left\{\left(1 + \frac{\alpha}{1 + \xi} \right)^k (1 + \alpha + \xi)^{2l} \right\}\right]_{\xi = 0} \notag \\
&= {2l\choose l} (1 + \alpha)^{k + l} + \sum^l_{j = 1} \mathcal T_j(\alpha), \notag
\end{align}
where
\beq
\mathcal T_j(\alpha) := {l\choose j}
\frac{(2l)!}{l! (l + j)!} (1 + \alpha)^{l + j} \left[\left(\frac{d}{d\xi}\right)^j \left(1 + \frac{\alpha}{1 + \xi}\right)^k \right]_{\xi = 0}.
\label{eq:6.16}
\eeq
We remark here that the simpler case $R_1 = R_2 \Leftrightarrow a = 1$ yielded immediately $\mathcal T_{q,1}(1) = {2l \choose l}$ in \cite{GMPY}.
Due to $j \geq 1$ the $j$\/{th} derivative has the form
\beq
\left(\frac{d}{d\xi}\right)^{j - 1} \left\{ k\alpha \left(1 + \frac{\alpha}{1 + \xi}\right)^{k - 1} \frac{(-1)}{(1 + \xi)^2} \right\},
\label{eq:6.17}
\eeq
so, after taking again derivatives $j - 1$ $(\geq 0)$ times we obtain
\beq
\mathcal T_j(\alpha) \ll \alpha \ \text{ for } \ j = 1,2, \dots, l
\label{eq:6.17masodik}
\eeq
which, finally, by \eqref{eq:6.15}--\eqref{eq:6.16}, yields
\beq
\mathcal T_{q,1}(a) = \mathcal T_{q,1}(1 + \alpha) = {2l\choose l} + O(\alpha).
\label{eq:6.18}
\eeq

Hence an easy calculation shows that the contribution of the residue at $s_1 = s_2 = 0$ in the value of $\wt{\mathcal T}_q$ in \eqref{eq:6.9} is
\beq
\ll \alpha \frac{G_q(0,0)}{(k + 2l)!} (\log R)^{k + 2l} {2l\choose l} + O\bigl((\log R)^{k + 2l - 1} \log^c_2 N\bigr),
\label{eq:6.19}
\eeq
where $\alpha = -\log q/\log R$ or $\log q/\log (R/q)$, and
\beq
G_q(0,0) = \mathfrak S(\mathcal H) \left(1 - \frac{|\Omega(q)|}{q} \right)^{-1} \ll \mathfrak S(\mathcal H),
\label{eq:6.20}
\eeq
since $|\Omega(q)| \leq \min (q - 1, k)$ because $\mathcal H$ is admissible.

We need also a more careful treatment in the error term estimation, since the earlier relation $R^{s_1 + s_2} \ll 1$ for $|s_1 + s_2| = (\log N)^{-1}$ (cf.\ (1.7)--(1.9) of \cite{GMPY}) is no longer true if $R^{s_1 + s_2}$ is replaced by $R^{s_1}_1 R^{s_2}_2$ with $R_1 \neq R_2$ ($R_1 = R$, $R_2 = R/q$ or reversed).

However, due to the symmetry of $s_1$ and $s_2$ in \eqref{eq:6.10} we may suppose during the estimation of \eqref{eq:6.10} that $R_2 \leq R_1$.
In this case (1.7) of \cite{GMPY} remains valid since in the shifted integral for $s_1 \in L_3$, $s_2 \in L_2$ ($L_2$ and $L_3$ are defined as in \cite{GMPY} by $L_2 : c_2 / (2\log U) + it$, $|t| \leq U/2$, $L_3 : -c_2/\log U + it$, $|t| \leq U$ with $U = \exp (\sqrt{\log N})$) we have, similarly to \cite{GMPY}, but now by $\Re\, s_2 > 0$
\beq
\bigl|{R_1}^{s_1} {R_2}^{s_2}\bigr| \leq \bigl|{R_1}^{s_1 + s_2}\bigr| \leq {R_1}^{-c_2/(2\log U)} \ll \exp \bigl(-c\sqrt{\log N}\bigr).
\label{eq:6.21}
\eeq

Further, in the secondary term, writing $C(s_2) : |s_1 + s_2| = (\log N)^{-1}$ we obtain again, with the same notation as in \cite{GMPY},
\begin{align}
I :&= \frac1{2\pi i} \intl_{L_2} \Bigl\{\Res_{s_1 = -s_2} \Bigr\} ds_2
\label{eq:6.22}\\
&= \frac1{2\pi i} \intl_{L_2} \intl_{C(s_2)} G_q(s_1, s_2) \left(\frac{\zeta(s_1 + s_2 + 1)}{\zeta(s_1 + 1) \zeta(s_2 + 1)} \right)^k \frac{R^{s_1}_1 R^{s_2}_2}{(s_1 s_2)^{k + l + 1}} ds_1 ds_2, \notag
\end{align}
and the crucial quantity $R^{s_1}_1 R^{s_2}_2$ is for $s_1 \in C(s_2)$, $s_2 \in L_2$
\beq
\bigl| R^{s_1}_1 R^{s_2}_2 \bigr| \leq \bigl|R^{s_1 + s_2}_1 \bigr| \leq e^{\log R_1 / \log N} < e,
\label{eq:6.23}
\eeq
since $R_2 \leq R_1$, $\text{\rm Re}\, s_2 = c_2 / (2\log U) > 0$, and everything else remains unchanged valid; yielding the same estimate as in (1.9) of \cite{GMPY}:
\beq
I \ll (\log N)^{k + l - 1/2} (\log_2 N)^c.
\label{eq:6.24}
\eeq

During the proof we used that the only difference between the present $G_q(\bold  s)$ and $G(\bold  s)$ of \cite{GMPY} (and similarly with $F$ and $Z$) is that (due to $|\Omega(q)| \leq \min(k, q - 1)$)
\beq
\frac{Z_q(\bold  s)}{Z(\bold  s)} = \frac{G_q(\bold  s)}{G(\bold s)} = \frac{F_q(\bold  s)}{F(\bold s)} = \left(1 - \frac{|\Omega(q)|}{q} \left( \frac1{q^{s_1}} + \frac1{q^{s_2}} - \frac1{q^{s_1 + s_2}}\right)\right) \ll 1
\label{eq:6.25}
\eeq
in both regions $\Re s_1, \Re s_2 \in (-c_3, c_3)$ and $\Re s_1, \Re s_2 \geq 0$, say, if $c_3 < c_0(k, l)$ is chosen sufficiently small.

Finally the last integral, $s_1 \in L_3$, $s_2 \in L_4 : - c_2/\log U + it$, $|t| \leq U/2$ remains to be $O\bigl(\exp(-c\sqrt{\log N})\bigr)$ as in \cite{GMPY}, which finishes the proof of Lemma~\ref{lem:3}.
Summation over all primes $q \leq R^\eta$ gives an upper estimate for the LHS of \eqref{eq:6.2} and yields Lemma~\ref{lem:4}, in view of \eqref{eq:6.5}, \eqref{eq:6.6}, \eqref{eq:6.19}--\eqref{eq:6.21}, \eqref{eq:6.24}, \eqref{eq:6.25}, since
\beq
\sum_{q\leq R^\eta} \frac{|\Omega(q)|}{q} \frac{\log q}{\log R} \ll \eta.
\label{eq:6.26}
\eeq
\end{proof}

The proof of Lemma~\ref{lem:5} runs again similarly to that of the analogous Lemma~2 of \cite{GMPY}.
The needed changes are essentially the same as described above, so we will be brief.
First we remark that during the applications of the Bombieri--Vinogradov theorem or its hypothetical improvement \eqref{eq:1.1} the existence of a prime $q \leq R^\eta$ with the extra condition $q \mid \mathcal P_{\mathcal H}(n)$ means that we need now the stronger condition
\beq
[q, d_1, d_2] \leq R^{2 + \eta} \leq N^{\vartheta - \ve},
\label{eq:6.27}
\eeq
which, however, appears in the statement of our Lemma~\ref{lem:5}.

With this change in our assumption the substitution of the contribution of the primes by its expected contribution, the analogues of formulae \eqref{eq:2.4}--\eqref{eq:2.6} and the displayed inequality following (2.6) in \cite{GMPY} remain valid with the change that $[d_1, d_2]$ has to be replaced always by $[q, d_1, d_2]$.
After this substitution, we arrive again at the analogous quantity
\beq
\prod_{p\mid [q, d_1, d_2]} \biggl(\sum_{b\in \Omega(p)} \delta  ((b + h, p))\biggr) = \bigl(|\Omega^+(q)| - 1\bigr) \prod_{\substack{p \mid [d_1, d_2]\\ p \neq q}} \bigl(|\Omega^+(p)| - 1\bigr),
\label{eq:6.28}
\eeq
where $\Omega^+$ corresponds to the set $\mathcal H^+ = \mathcal H \cup \{h\}$ (and as remarked in \cite{GMPY}, $|\Omega^+(p)| = p$ can occur now already) and $\delta(m) = 1$ if $m = 1$, $\delta(m) = 0$ if $m \neq 1$.

This yields now, similarly to \eqref{eq:6.5}--\eqref{eq:6.6} to the slightly modified analogue of $\mathcal T^*$ in (2.7) of \cite{GMPY}, to the expressions
\beq
\frac{N\bigl(|\Omega^+(q)| - 1\bigr)}{q - 1} \mathcal T^*_q
\label{eq:6.29}
\eeq
with
\beq
\mathcal T^*_q = \frac1{(2\pi i)^2} \intl_{(1)} \intl_{(1)} \wt F^*_q (s_1, s_2) \frac{R^{s_1 + s_2}}{(s_1 s_2)^{k + l + 1}} ds_1 ds_2,
\label{eq:6.30}
\eeq
where, with the same $F^\#_q(\bold  s) = (1 - q^{-s_1})(1 - q^{-s_2})$ as in \eqref{eq:6.7} we have now
\beq
\wt{\mathcal F}^*_q = F^\#_q \cdot F^*_q, \quad
F^*_q(\bold  s) = \prod_{p \neq q} \left(1 - \frac{|\Omega^+(p)| - 1}{p - 1} \left( \frac1{p^{s_1}} + \frac1{p^{s_2}} - \frac1{p^{s_1 + s_2}}\right)\right).
\label{eq:6.31}
\eeq

The whole treatment of the error terms is the same as in case of Lemma~\ref{lem:3}, the only change being in the main term and in the singular series.
We have to distinguish two cases (although as mentioned earlier Case~2 is not needed for the present work).

\smallskip
\emph{Case 1}. $h \in \mathcal H$.
In this case $\mathcal H^+ = \mathcal H \cup \{ h\} = \mathcal H$, $m = 1$, $\Omega^+(d) = \Omega(d)$ for every $d$, the singular series is according to \cite{GMPY} \eqref{eq:6.25} and \eqref{eq:6.31} now
\beq
G^*_q(0,0) = \mathfrak S_q(\mathcal H^+) = \prod_{p\neq q} \left(1 - \frac{|\Omega^+(p)|}{p}\right) \left(1 - \frac1p\right)^{-(k + 1)} \ll \mathfrak S(\mathcal H^+) = \mathfrak S(\mathcal H)
\label{eq:6.32}
\eeq
and the same reasoning as in \cite{GMPY}, the translation $k \to k - 1$, $l \to l + 1$ gives the result, since, if $n + h \in \mathcal P$, then
\beq
d \mid P_{\mathcal H}(n) \Longleftrightarrow d \mid P_{\mathcal H \setminus\{ h\}}(n) .
\label{eq:6.33}
\eeq

\smallskip
\emph{Case 2}. $h \notin \mathcal H$.
In this case $\mathcal H^+ = \mathcal H \cup \{h\}$, $m = 0$, $\Omega^+(p) = k + 1$ for $p > k$ and $\mathcal H^+$ is not necessarily admissible.
However, $\Omega^+(p) = p$ may occur only for $p \leq k + 1$ since $\Omega^+(p) \leq k + 1$.
If $\Omega^+(p) = p$ is the case for some $p \neq q$, then as remarked in \cite{GMPY} the corresponding Euler product vanishes at $s_1 = 0$ or $s_2 = 0$, the main term lacks $\bigl(G^*_q(0,0) = \mathfrak S_q(\mathcal H^+) = 0\bigr)$ and the error term is the same or actually smaller.
Finally if the only prime for which $\Omega^+(p) = p$ holds is $p = q \leq k + 1$, then by $\mathfrak S$ $(\mathcal H \cup \{h\}) = 0$ we use the trivial consequence of Lemma~\ref{lem:2}:
\begin{align}
\label{eq:6.34}
\sum_{\substack{n \sim N\\
(P_{\mathcal H}(n), P(R^\eta)) > 1}} \theta(n + h) \Lambda_R(n; \mathcal H, k + l)^2
&\leq \sum_{n \sim N} \theta(n + h) \Lambda_R(n; \mathcal H, k + l)^2\\
&\ll N(\log R)^{k + 2l - 1} (\log_2 N)^c.\notag
\end{align}

\smallskip
\noindent
{\bf Remark.}
In most applications we can replace Lemma~\ref{lem:5} with the following slightly weaker assertion, which is a trivial consequence of Lemma~\ref{lem:4}.
\begin{lemma}
\label{lem:6}
Let $N^{c_0} < R \leq N^{1/2}(\log N)^{-C}$.
Then we have for any $h \leq H$
\beq
\sum_{\substack{n \sim N\\
(P_{\mathcal H}(n), P(R^\eta)) > 1}} \theta(n + h) \Lambda_R(n; \mathcal H, k + l)^2 \ll \eta \log N \sum_{n \sim N} \Lambda_R(n; \mathcal H, k + l)^2.
\label{eq:6.35}
\eeq
\end{lemma}

Lemma~\ref{lem:6} can relatively well substitute for Lemma~\ref{lem:5} if $h \in \mathcal H$ and the dependence of the constants on $k$ and $l$ in the $\ll$ symbol is not investigated, since by Lemma~\ref{lem:2} we have in fact for $h \in \mathcal H$
\beq
\log R \sum_{n \sim N} \Lambda_R(n; \mathcal H; k + l)^2 \sim C(k,l) \sum_{n \sim N} \theta(n + h) \Lambda_R(n; \mathcal H, k + l)^2
\label{eq:6.36}
\eeq
with a constant $C(k,l)$ depending only on $k$ and $l$, so the right-hand  sides of \eqref{eq:6.35} and \eqref{eq:6.3} are really the same order of magnitude as a function of~$R$ and~$N$.

\section{The execution of Step 2. Partial proof of Theorem~\ref{th:3}}
\label{sec:7}

As mentioned already in the previous section, Lemma~\ref{lem:4} and its trivial consequence Lemma~\ref{lem:6}, together with Lemmas~\ref{lem:1} and \ref{lem:2} contain already sufficient information about primes in almost prime $k$-tuples, needed to prove later Theorems~\ref{th:1} and~\ref{th:2}.

We have, namely, similarly to \eqref{eq:5.8}, by Lemmas \ref{lem:1}, \ref{lem:2}, \ref{lem:4} and \ref{lem:6}, for $R = N^{(\vartheta - \ve)/(2 + \eta)} > (3N)^{1/4}$
\begin{align}
\label{eq:7.1}
&\sum_{\substack{n \sim N\\
(P_{\mathcal H}(n), P(R^\eta)) = 1}} \biggl(\sum_{h \in \mathcal H} \theta(n + h) - \log 3N\biggr) \Lambda_R(n; \mathcal H, k + l)^2\\
&= \frac{\mathfrak S(\mathcal H)}{(k + 2l)!} {2l\choose l} N \log N
(\log R)^{k + 2l} \times\notag\\
&\quad \times\left( \frac{k}{k + 2l + 1} \cdot \frac{2(2l + 1)}{l + 1} \cdot \frac{\vartheta - \ve}{2 + \eta} + o(\eta) - 1 + o(1)\right).\notag
\end{align}
It is easy to see that for any given $\vartheta = \frac12 + \delta$, if $l$ and $k/l$ are chosen sufficiently large, then
\beq
\frac{k}{k + 2l +1} \cdot \frac{2l + 1}{l + 1} \left(\frac12 + \delta\right) > 1.
\label{eq:7.2}
\eeq
Now we can choose $\ve$ and $\eta = c(k, l, \vartheta)$ sufficiently small as to have
\begin{align}
\label{eq:7.3}
&\sum_{\substack{n\sim N\\
(P_{\mathcal H}(n), P(R^\eta)) = 1}} \biggl(\sum_{h \in H} \theta(n + h) - \log 3N\biggr)\Lambda_R(n; \mathcal H, k + l)^2\\
 & \hspace*{30mm} \gg_{k, l, \mathcal H, \vartheta} \, N \log N(\log R)^{k + 2l}.
\notag
\end{align}

However, if $\bigl(P_{\mathcal H}(n), P(R^\eta)\bigr) = 1$, $R > (3N)^{1/4}$, then any $P_{\mathcal H}(n)$ has at most $k \cdot \frac{4}{\eta}$ prime divisors, so we have
\beq
|\Lambda_R(n; \mathcal H, k + l)|^2 \leq \left(\frac{2^{4k/\eta}}{(k + l)!} (\log R)^{k + l} \right)^2.
\label{eq:7.4}
\eeq
Now if we have at most one prime among $n + h_i$ $(i = 1,2, \dots, k)$, then $\sum_{h \in \mathcal H} \theta(n + h) - \log 3N < 0$, so we obtain for the number of $n$'s in $[N, 2N]$ with at least two primes among $(n + h_i)$ and almost primes in all coordinates $n + h_j$ with $P^-(n + h_j) > n^{1/4\eta}$ the lower estimate
\beq
c(k,l, \mathcal H, \vartheta) \frac{N}{(\log R)^k} > c'(k, l, \mathcal H, \vartheta) \frac{N}{(\log N)^k}
\label{eq:7.5}
\eeq
as required by \eqref{eq:4.1} of Theorem~\ref{th:3}.
We remark that the dependence on $l$ and $\mathcal H$ can be omitted, since for $k \to \infty$ we will choose $l = (\sqrt k - 1)/2$ (cf.\ \eqref{eq:8.1}), further we have for any admissible $k$-tuple~$\mathcal H$
\begin{align}
\label{eq:7.6}
\mathfrak S(\mathcal H) :&= \prod_p \left(1 - \frac{|\Omega(p)|}{p} \right) \left(1 - \frac1p\right)^{-k} \\
&\geq \prod_{p \leq 2k} \frac1p \cdot \prod_{p > 2k} \left(1 - \frac{k}{p}\right) \left(1 - \frac1p\right)^{-k} \geq c_3(k). \notag
\end{align}

The extra assertion that we have at least two primes $n + h_i$ and $n + h_j$ in some position $(i,j)$ and we have the same number $b_s$ of prime divisors of $n + h_s$ ($s \neq i,j$, $1 \leq s \leq k$) for all elements $n$ of the progression, is a trivial consequence of the fact that the number of the possible vectors $\bold  b = (b_1, \dots, b_k)$ is bounded (by $(1 / c_1(k))^k$) if all $n + h_s$ components are free of prime factors below $n^{1/c_1(k)}$.
This means that at least one configuration, that is, one vector $\bold  b$ (with at least two entries equal to~$1$) occurs at least $c_4(k) N/\log^k N$ times, fully describing the multiplicative pattern of $n + \mathcal H$ by $\Omega(n + h_s) = b_s$, where differently from the previous two sections $\Omega(n)$ denotes here the number of prime divisors of~$n$.
We may mention that we could require beyond $\Omega(n + h_s) = b_s$ for all $n$ also the stronger property that the exponent  pattern $A_s = \bigl\{\alpha_{s1}, \alpha_{s2}, \dots, \alpha_{sj_s}\bigr\}$ of $n + h_s$ should be the same for all elements $n$ of the progression.
Namely, due to the trivial relation $b_s = \alpha_{s1} + \dots + \alpha_{sj_s}$, any vector $\bold  b = (b_1, \dots, b_k)$ gives rise only to a bounded number of possibilities for the values $\alpha_{st}$.
Hence at least one of them has to appear at least $c_4(k) N/(\log N)^k$ times (with a different value of $c(k)$, however) for $n$'s up to~$N$.

\section{How to choose the parameters $k, l$ and a small $\mathcal H_k$ for a given distribution level~$\vartheta$? Continuation of the proof of Theorem~\ref{th:3}}
\label{sec:8}

In order to prove Theorem~\ref{th:B}, further our present Theorems~\ref{th:1} and \ref{th:2} the values of the parameters $k, l$ could be optimized to yield a minimal $k$ for a given $\vartheta = 1/2 + \delta > 0$ by the aid of computers as long as $\delta$ is not too small
($\delta \geq 1/10$, for example, see the table after (3.4) on p.~832 in \cite{GPY1}).
The crucial inequality to be satisfied is our \eqref{eq:7.2}.

In view of the above, we will focus our attention to small values of $\delta$ (which means large values of $k$ and $l$), although our argument holds for any $\delta \in (0, 1/2]$.
An easy calculation gives that if we did not require $l$ to be an integer, then for a given $k$ the expression on the left-hand side of \eqref{eq:7.2} would be maximal for $l = (\sqrt{k} - 1)/2$, i.e.\ $k = (2l + 1)^2$ and then its value is for $l \geq (2\delta)^{-1}$
\beq
\geq \frac{(2l + 1)^2}{2(l + 1)(2l + 1)} \cdot \frac{2l + 1}{l + 1} \cdot \frac{l + 1}{2l} = \frac{k}{k - 1}.
\label{eq:8.1}
\eeq

We remark that if $\delta = 1/2$ ($\vartheta = 1$) for example, then this argument would give $l = 1$, $k = 9$, whereas $l = 1$, $k = 7$, $\vartheta > 20/21$ already satisfies \eqref{eq:7.2}.
A further improvement is in this case (at least for $\vartheta > 4(8 - \sqrt{19})/15 = 0.97096\dots$) possible by choosing instead of the single optimal $l = 1$ a linear combination of the weight functions $\Lambda_R(n; \mathcal H, k + l)$ for $l = 0$ and $l = 1$. Then the argument works for $k = 6$ already as shown in Section~{3} (cf.\ (3.11)--(3.16)) of \cite{GPY1}.
Since all our earlier arguments remain valid if instead of a single weight function $\Lambda_R(n; \mathcal H, k + l)$ we choose a linear combination of them, the arguments (6.11)--(6.16) of \cite{GPY1} together with our present ones in Sections \ref{sec:5}--\ref{sec:7} prove Theorems~\ref{th:1}--\ref{th:4} for $\vartheta \geq 0.971$.

In order to construct an admissible $k$-tuple $\mathcal H = \{h_i\}^k_{i = 1}$ with a possibly small diameter $h(k) := h_k - h_1$, we can again obtain help from computers for relatively small values of $k$ (cca.\ $k < 100$) as shown by the table after (3.4) in \cite{GPY1}.
However, for any value of $k$ we can choose $\mathcal H$ as the first $k$ primes exceeding $k$, $\{p_{\nu + 1}, \dots, p_{\nu + k} \}$, where $p_\nu \leq k < p_{\nu + 1}$.
This set clearly does not cover the residue class $0$ for $p \leq k$, while for $p > k = |\mathcal H|$ it clearly can not cover all residue classes $\mod p$.
On the other hand the diameter of $\mathcal H$ is by the prime number theorem, that is, by $p_n \sim n \log n$
\beq
h(k) = p_{\nu + k} - p_{\nu + 1} = (1 + o(1)) \left\{ \left( k + \frac{k}{\log k} \right) \log k - k\right\} \sim k \log k
\label{eq:8.2}
\eeq
if $k \to \infty$ (which occurs for $\delta \to 0$).
This is asymptotically probably close to optimal, since in general a set $\mathcal H$ of numbers up to $X$ avoiding at least one residue class $\mod p$ for any $p \leq z$ is heuristically of size at most
\beq
X \prod_{p \leq z} \left(1 - \frac1p\right) \sim X \frac1{e^\gamma \log z},
\label{eq:8.3}
\eeq
whereas our set above has a somewhat larger density $\sim 1/\log p_{\nu + k} \sim 1/\log h_k$.

\section{How do we get strong de Polignac numbers? Completion of the proof of Theorem~\ref{th:3}}
\label{sec:9}

In this section we will show that we obtain at least $c_2(k) N/\log^k N$ numbers $n$ up to~$N$, where $n + h_i$ and $n + h_j$ are consecutive primes.
Let with a fixed sufficiently small $c_1(k)$
\beq
\mathcal B(i, j, N) = \bigl\{n \leq N; \ n + h_i \in \mathcal P, \ n +h_j \in \mathcal P, \ P^-(P_{\mathcal H}(n)) > n^{c_1(k)}\bigr\},
\label{eq:9.1}
\eeq
\beq
\mathcal T = \left\{(i, j);\ j > i, \ \liminf_{N \to \infty} \frac{|\mathcal B(i,j,N)\log^k N|}{N} > 0 \right\},
\label{eq:9.2}
\eeq
and let us choose any given pair $\{s,t\} \in \mathcal T $ with minimal value of $t - s$.
Then for any $h_\mu \in (h_s, h_t)$ we must have clearly
\beq
\liminf_{N\to\infty} \frac{|\mathcal B(\mu, t, N)| \log^k N}{N} = 0,
\label{eq:9.3}
\eeq
so all components $n + h_\mu$ between $n + h_s$ and $n + h_t$ are almost always composite if $n \in \mathcal B(s, t, N)$ as $N \to \infty$.

On the other hand, if we have an arbitrary $h \in (h_s, h_t)$, $h \notin \mathcal H$, then the assumption $n + h \in \mathcal P$ implies for $\mathcal H^+ = \mathcal H \cup h$
\beq
P^-\bigl(P_{\mathcal H^+} (n)\bigr) > n^{c_1(k)}.
\label{eq:9.4}
\eeq
However, by the Selberg sieve (cf.\ Theorem 5.1 of \cite{HR}, or alternatively Theorem~2 of Section 2.22 of \cite{Gre}, or our present Lemma~\ref{lem:1}, the number of such $n \leq N$
\beq
\ll_{k, c_1} \frac{\mathfrak S(\mathcal H \cup \{h\}) N}{\log^{k + 1}N} \ll_{k, c_1} \frac{\mathfrak S(\mathcal H) N \log h_k}{\log^{k + 1}N} \ll_{k, c_1, \mathcal H} \frac{N}{\log^{k + 1}N},
\label{eq:9.5}
\eeq
which means that for a given fixed $\mathcal H$, this case might happen also rarely.
This, together with \eqref{eq:9.3} shows that the number of $n \leq N$ where Theorem~\ref{th:3} is true with two consecutive primes is, similarly to \eqref{eq:4.1}, at least
\beq
\bigl(c_1(k, \mathcal H) + o(1)\bigr) \frac{N}{\log^k N}.
\label{eq:9.6}
\eeq

\section{Application of the method of Green and Tao. Proofs of Theorem~\ref{th:1} and \ref{th:2}}
\label{sec:10}

Since Theorem~\ref{th:2} is a more general form of Theorem~\ref{th:1}, it is clearly sufficient to prove just Theorem~\ref{th:2}.
This is relatively easy and straightforward as we now already proved Theorem~\ref{th:3}.
So we have for any admissible $\mathcal A$ with $r \geq (2 \lceil 1/2\delta \rceil + 1)^2$ elements a set $\mathcal N^*(\mathcal A) = \mathcal N^* \subset \mathbb N$ at our disposal with the properties that with some $i,j \in \{1, \dots, r\}$ and some $b_s \leq C(r)$ we have for $n \in \mathcal N^*$
\begin{gather}
\label{eq:10.1}
n + a_i \ \text{ and }\ \ n + a_j \ \text{ are consecutive primes},\\
\label{eq:10.2}
\Omega(n + a_s) = b_s, \ P^-(n + a_s) \geq n^{c_1(r)} \text{ for } s \in \{1, \dots, r\},\\
\label{eq:10.3}
\bigl|\{n \leq N; \ n \in \mathcal N^*\}\bigr| \geq c_1(r, \mathcal A) \frac{N}{\log^r N}.
\end{gather}
As remarked at the end of Section~\ref{sec:7} the condition $\Omega(n + a_s) = b_s$ might be even replaced by the stronger condition that the exponent pattern of $n + a_s$ should be $\boldsymbol {\alpha}_s = \bigl\{\alpha_{s1}, \ldots, \alpha_{sj_s}\bigr\}$.

This set $\mathcal N^*$ has a positive lower density in the set $\wt{\mathcal N}$ of all integers satisfying
\beq
P^- \bigl(P_{\mathcal A}(n)\bigr) \geq n^{c_1(r)},
\label{eq:10.4}
\eeq
due to the already mentioned Theorem~5.1 of \cite{HR}, Theorem 2.2.2.2 of \cite{Gre}, or our Lemma~\ref{lem:1} (cf.\ \eqref{eq:9.4}--\eqref{eq:9.5}).

\smallskip
\noindent
{\bf Remark.}
The above formulation shows that the generalization of the somewhat heuristic description, appearing in many works of Green and Tao that during their proof the primes are embedded into the set of almost primes with positive (lower) density can be proved in an exact form (\eqref{eq:10.1}--\eqref{eq:10.4}) in our case as well.

\medskip
The proof now follows closely that of Green and Tao (cf.\ Sections 9--10 and the Appendix of \cite{GT}).
Our task is made even easier by the recent work of Binbin Zhou \cite{Zho}, where he proved the existence of arbitrary long arithmetic progressions of Chen primes, where for the sake of convenience he defined $p$ to be a Chen prime if
\[
p \in \mathcal P, \ \Omega(p + 2) \leq 2, \  P^-(p + 2) \geq P^{1/10}
\]
and used the lower bound $CN/\log^2 N$ for the number of Chen primes below~$N$.

\medskip
In fact we can formulate our result in the following general form.

\begin{theorem}
\label{th:5}
Let $\mathcal A = \{a_1, \dots, a_r\} \subseteq [0,A] \cap \mathbb Z$, $\mathcal P_{\mathcal A}(n) = \prod\limits^r_{i = 1}(n + a_i)$.
Let $P^-(n)$ denote the least prime divisor of~$n$.
Let the set $\mathcal N(\mathcal A)$ satisfy
\beq
\mathcal N(\mathcal A) \subseteq \bigl\{n; P^-(P_{\mathcal A}(n)) \geq n^{c_1}\bigr\}, \quad \bigl|\{n \leq X;\ n \in \mathcal N(\mathcal A)\}\bigr| \geq \frac{c_5 X}{\log^r X},
\label{eq:10.5}
\eeq
with $c_1, c_5 > 0$ for $X > X_0$.
Then $\mathcal N(\mathcal A)$ contains $m$-term arithmetic progressions for any $m > 0$.
\end{theorem}

\noindent
{\bf Remark.}
This is clearly a generalization of the results of Green--Tao ($r = c_1 = 1$) and Zhou ($r = 2$, $c_1 = 1/10$).

\smallskip
\noindent
{\bf Remark.}
The terms of the arithmetic progression of length $m$ are below $N$ if $N > N_0(c_1, c_2, r, A, m)$ and their total number is at least $c_3(r, A, m) N^2/\\
/\log^{rm} N$.

\smallskip
\noindent
{\bf Remark.}
Theorem~\ref{th:5} trivially shows that the twin primes really contain arbitrarily long arithmetic progressions if their number up to $x$, $\pi_2(x) \gg x / \log^2 x$.
This result is implicitly contained in \cite{Zho} as well.

\smallskip
\noindent
{\bf Remark.}
$P^-\bigl(P_{\mathcal A}(n)\bigr) \geq n^{c_1}$ implies that $\mathcal A$ is admissible, since otherwise $P_{\mathcal A}(n)$ would have a fixed prime divisor~$p \leq r$.

\smallskip
\noindent
{\bf Remark.}
Since the proof is analogous to that in \cite{Zho}, which in fact is analogous to that in \cite{GT} we will point out only the essential differences.

\smallskip
\noindent
{\bf Remark.}
The above said strong analogy is only true if $\mathcal A$ is considered to be fixed, more precisely if
\beq
\mathcal A \subseteq [1, A] \quad \text{ with a bounded }\ A.
\label{eq:10.6}
\eeq
Otherwise, when $A$ is allowed to increase with $N$, serious difficulties may occur with the linear form property.

\smallskip
As the reader observed we changed our set $\mathcal H$ to $\mathcal A$, the elements $h_i$ to $a_i$ and the size $k$ to $r$, compared with Sections~\ref{sec:1}--\ref{sec:9}.
This is necessary since $\mathcal H$ and $k$ are used in \cite{GT} (and \cite{Zho}) to denote other quantities, namely our aim is to show the existence of $k$-term arithmetic progressions in~$\mathcal N^*$.

The definition of $W = W(N)$ and $w = w(N)$ remain the same,
\beq
W = \prod_{p \leq w} p,
\label{eq:10.7}
\eeq
where $w = w(N)$ and thereby $W = W(N)$ is a function of~$N$, sufficiently slowly growing to infinity with~$N$.

In the following we will suppose that $n \in \mathcal N^*$.
In this case $(n + a_i, W) = 1$.
Following Zhou (Section~\ref{sec:2}) we will choose a $b \mod W$ with $(Wm + b + a_i, W) = 1$, through first choosing a $b_p \mod p$ for every prime $p \mid W$ with $b_p \not\equiv - a_i(\mod p)$ and then applying the chinese remainder theorem to obtain $b \equiv b_p(\mod p)$ for each $p \mid W$.
Since we have for any $p$ exactly $p - |\Omega_{\mathcal A}(p)|$ possibilities for $b_p$ (where $\Omega_{\mathcal A}(p) = \{-a_i \mod p\}^r_{i = 1}$, as in Section~\ref{sec:5}), we obtain for the cardinality of the set $X_W$ of possible choices of $b \mod W$ the quantity (we note that we can suppose $r < w \to \infty$)
\begin{align}
\label{eq:10.8}
|X_W|
&= W \cdot \prod_{p \leq A} \left(1 - \frac{|\Omega_{\mathcal A}(p)|}{p} \right) \prod_{A < p \leq w(n)} \left(1 - \frac{r}{p} \right) \\
&\leq C(A) \prod_{p \mid W} \left(1 - \frac1{p}\right)^r = C(A) \left(\frac{\varphi(W)}{W} \right)^r
\notag
\end{align}
for every admissible set $\mathcal A$ with a uniform constant $C(A)$.
If we do not indicate further dependence on either $A$, $r$ or $k$, then we obtain by \eqref{eq:10.3}
\beq
\sum_{b \in X_W} \bigl| n \in [\ve_k N, 2 \ve_k N]; \ Wn + b \in \mathcal N^*\bigr| \gg \frac{N}{\log^r N},
\label{eq:10.9}
\eeq
since $P^-\bigl(P_{\mathcal A}(Wn + b)\bigr) > n^c > W$ implies $(Wn + b + a_i, W) = 1 \Leftrightarrow b \in X_W$.
Thus by \eqref{eq:10.8} we can choose a fixed residue class $b \mod W$
(depending on $\mathcal A$), $0 \leq b < W$, such that the set
\beq
|X| := \bigl|\bigl\{n \in [\ve_k N, 2 \ve _k N]\bigr\}; \ Wn + b \in \mathcal N^*\bigr| \gg \frac{\ve_k N}{\log^r N} \cdot \left(\frac{W}{\varphi(W)}\right)^r,
\label{eq:10.10}
\eeq
where $\ve_k = 1/(2^k(k + 4)!)$.
Our measure $\nu$ is now, similarly to (2.5) of \cite{Zho} defined on $\mathbb Z_N$ by
\beq
\nu(n) := \begin{cases}
\left(\dfrac{\varphi(W)}{W}\right)^r \prod\limits^r_{i = 1} \dfrac{{\Lambda_R}(W n + b + a_i)^{2}}{\log R} & \text{if } \ n \in [\ve_k N, 2\ve_k N]\\
1 & \text{otherwise}
\end{cases},
\label{eq:10.11}
\eeq
where $\Lambda_R$ is given in \eqref{eq:2.1}, $R$ will be chosen as a sufficiently small power $(< c(r, m))$ of $N$, thereby satisfying
\beq
(Wn + b + a_i, R) = 1 \quad \text{ for } \ n \in X;
\label{eq:10.12}
\eeq
due to the crucial condition \eqref{eq:10.2} and the definition of $X$ in \eqref{eq:10.10}, taking into account that $W \ve_k > 1$ by $W \to \infty$.
Now \eqref{eq:10.12} implies trivially for $n \in [\ve_k N, 2 \ve_k N] \cap {\mathcal N}^*$
\beq
\Lambda_R(Wn + b + a_i) = \log R, \ \text{ so } \ \nu(n) = \left(\frac{\varphi(W)}{W} \right)^r (\log R)^r .
\label{eq:10.13}
\eeq
This means that defining (in analogy with (2.1) of \cite{Zho})
\beq
\wt \Lambda_{\mathcal A}(n) := \begin{cases}
\frac{\varphi(W)}{W} \log (Wn + b) &\text{if }\ Wn + b \in \mathcal N^*\\
0 & \text{otherwise}
\end{cases} ,
\label{eq:10.14}
\eeq
we have
\beq
\nu(n) \geq f(n) := \bigl[k^{-1} 2^{-k - 5} \wt\Lambda_{\mathcal A}(n)\bigr]^r \ \text{ for } \ n \in [\ve_k N, 2\ve_k N].
\label{eq:10.15}
\eeq

The proof that $\nu(n)$ satisfies the $k$-pseudorandomness property follows that of \cite{Zho}, which again follows the proof of \cite{GT} which are essentially the special cases $r = 2$ and $r = 1$ of our case.
The fact that $\nu$ is a measure, that is,
$E(\nu) = 1 + o(1)$ is the special case of the linear form property ($m = 1$, $b = 1$, $\psi_1(x) = x_1$, $B = [\ve_\kappa N, 2 \ve_\kappa N]$).
The proof of the linear form condition runs completely analogously to that of \cite{Zho} (which is nearly the same as that of \cite{GT}), a crucial point being here that if a prime $p > W$ would satisfy
\beq
p \mid W \biggl( \sum^t_{l = 1} L_{il} x_l\biggr) + b + a_u \ \text{ and } \ p \mid W \biggl(\sum^t_{l = 1} L_{il} x_j\biggr) + b + a_v,
\label{eq:10.16}
\eeq
then obviously $p \mid a_u - a_v$, which is a contradiction since
\beq
\mathcal A = \{a_1, \dots, a_r\} \subseteq [1, A] \quad A \text{ is bounded, } \ W = W(N) \to \infty \ \text{ as } N \to \infty.
\label{eq:10.17}
\eeq
Hence, for $u \neq v$ \ $\theta_i(\bold  x) + a_u \equiv \theta_i(\bold  x) + a_v \equiv 0$ $(\mod p)$ is impossible.
(This step is not valid if we allow $A \to \infty$ with $A \geq W(N)$.)
Concerning the analogue of Proposition~2.10 of \cite{Zho}, or (9.10) of \cite{GT} according to which the measure $\nu$ satisfies the $2^{k - 1}$-correlation condition, the main difference is that the role of $\Delta$ is played here (cf.\ Proposition~2.6 of \cite{Zho} or Proposition~9.6 of \cite{GT}) by the quantity
\beq
\Delta = \prod_{1 \leq i < j \leq m} (h_i - h_j) \prod_{1 \leq u < v \leq r} \bigl(W(h_i - h_j) + a_u - a_v\bigr) \ll N^{{m\choose 2}{r \choose 2}}
\label{eq:10.18}
\eeq
which, however, still obeys the estimate (10.15) of \cite{GT}, namely
\beq
\Delta \ll R^{O_{m,r}(1)}.
\label{eq:10.19}
\eeq
Afterwards, the analogue of Lemma 2.9 of \cite{Zho} or Lemma 9.9 of \cite{GT} is here again slightly more difficult, in the sense that we need the generalized H\"older's inequality instead of the standard case, applied in \cite{Zho}, to infer that with the notation $I = [1,N]$, $S(n) = \prod\limits_{1 \leq u < v \leq r} (Wn + a_u - a_v)$ we have
\beq
E \biggl( \prod_{p \mid nS(n),\, p > w} \bigl(1 + p^{-1/2}\bigr)^{O_{m,r(q)}} \mid I \biggr) = O_{m,r,q}(1).
\label{eq:10.20}
\eeq
The quantity on the left-hand side is, namely, with the notation $B = {r \choose 2} + 1$
\beq
\leq E \biggl( \prod_{\substack{p \mid n\\ p > w}} \bigl(1 + p^{-\frac18}\bigr) \mid I \biggr)^{1/B} \prod_{1 \leq u < v \leq r} E \biggl( \prod_{p \mid Wn + a_u - a_v} \bigl(1 + p^{-\frac18}\bigr) \mid I \biggr)^{1/B}
\label{eq:10.21}
\eeq
and then the rest of Lemma 2.9 of \cite{Zho} can be followed using again the crucial property~\eqref{eq:10.17}.

So we obtain, similarly to \cite{Zho} and \cite{GT} for every value of $m$ actually $cN^2 / (\log N)^{rm}$ $m$-term arithmetic progressions in the set $\mathcal N^*(\mathcal A)$, where $c$ depends on $r$, $a_r$,  $m$, and $\mathcal A = \{a_i\}^r_{i = 1}$.
This proves our Theorem~\ref{th:5}, and consequently Theorems~\ref{th:1} and \ref{th:2}, too.

\section{The density of the de Polignac numbers}
\label{sec:11}

Let us suppose that any admissible $k$-tuple $\mathcal H$ produces at least two consecutive primes infinitely often.
Then this phenomenon clearly occurs infinitely often for prime pairs $n + h_i$, $n + h_j$ in the same position, i.e.\ we have at least one strong de Polignac number $d$ among $h_i - h_j$ $(i > j)$.
The question is: how many different $d$'s do we get at least by choosing all possible admissible sets $\mathcal H$ with elements at most~$N$, if $N$ is large $(N \to \infty)$.
Let
\beq
P := P(k) := \prod_{p \leq k} p,
\label{eq:11.1}
\eeq
where we can clearly suppose $P|N$, as the size of $k$ and thereby $P$ remains fixed and $N \to \infty$.

In order to show admissibility we may suppose that we choose all elements of all $k$-tuples $\mathcal H$ from the set
\beq
\mathcal M := \{m \leq N; \ (m, P) = 1\}, \ \text{ where } \ M := |\mathcal M| = \frac{\varphi(P)}{P} N;
\label{eq:11.2}
\eeq
thereby excluding the residue class $0$ for each $p \leq k$.
Since every admissible $k$-tuple $\mathcal H$ gives rise to at least one strong de Polignac number $d$, we obtain at least ${M \choose k}$ strong de Polignac numbers below $N$, counted with multiplicity according to the $k$-tuples $\mathcal H$.
A fixed value $d$ might appear as the difference of two elements of $\mathcal H$ with at most $M - 1$ choices for the larger element and afterwards the smaller element is determined uniquely.
Furthermore, we have ${M - 2\choose k - 2}$ choices for the remaining $k - 2$ elements of~$\mathcal H$.
This implies that we obtain at least
\beq
\frac{{M\choose k}}{(M - 1) {M - 2\choose k - 2}} = \frac{M}{k(k - 1)} = \frac{N \varphi(P)}{Pk(k - 1)}
\label{eq:11.3}
\eeq
different strong de Polignac numbers $d$ until $N$, which proves~\eqref{eq:4.8}.

In case of $\vartheta \geq 0.971$ we can work with $k = 6$ tuples, so $P = 30$, $\varphi(30) = 8$, which proves \eqref{eq:4.9}, hence Theorem~\ref{th:4} is proved completely.

\section{Further problems}
\label{sec:12}

It is clear from our work that Steps 1 and 2 (Lemmas \ref{lem:3}--\ref{lem:6}) led to some new information about primes in tuples, in particular about the frequency of the occurrence of two primes in any admissible $k$-tuple for $\vartheta > 1/2$ (Theorem~\ref{th:3}) and that under the same condition we have more than one even~$d$, in fact a positive proportion of all numbers, which appear infinitely many times as the difference of two consecutive primes.
We mentioned that Lemma~\ref{lem:4} might help to deduce unconditionally that a positive proportion of gaps between consecutive primes are less than $\eta \log p$ for any fixed $\eta > 0$.
The question still arises: does the combination of the two methods of \cite{GT} and \cite{GPY} yield also some unconditional results?
The answer is yes.
We mention a few of them.

\begin{theorem}
\label{th:6}
Let $\mathcal H = \{h_i\}^k_{i = 1}$ be an admissible $k$-tuple for any $k \geq 1$.
Then there exist arbitrarily long arithmetic progressions of primes $n$ such that all $n + h_i$'s are almost primes and with some vectors $\boldsymbol \alpha_i = (\alpha_{i1}, \dots, \alpha_{ij_i})$
\beq
P^-(n + h_i) > n^{c_1(k)}, \ \ n + h_i \ \text{ has exponent pattern }\ \boldsymbol \alpha_i,
\label{eq:12.1}
\eeq
for all values of $n$ in the progression.
(Here $b_i = \sum_{1 \leq s \leq j_i} \alpha_{is} \leq 1/c_1(k)$.)
\end{theorem}

It is not a consequence of the stated results but using the method of \cite{GGPY1--3} in combination with that  of \cite{GT} one can show the following unconditional results.

\begin{theorem}
\label{th:7}
Let $\mathcal H = \{h_i\}^k_{i = 1}$ be any admissible $k$-tuple for any $k \geq 3$.
Then there exist arbitrarily long arithmetic progressions of $n$ values such that for some $i,j \in \{1, \dots, k\}$, $i \neq j$, $n + h_i$ and $n + h_j$ are semiprimes (that is product of two different primes, i.e.\ having exponent pattern $(1,1)$) and some vectors $\bold \alpha_s$ for $s \in \{1, \dots, k\}$, $s \neq i, j$ such that
\beq
P^-(n + h_s) > n^{c_1(k)}, \ \ n + h_s \ \text{ has exponent pattern } \ \boldsymbol \alpha_s.
\label{eq:12.2}
\eeq
\end{theorem}

\noindent
{\bf Corollary.}
{\it There exist arbitrarily long arithmetic progressions of generalized twin semiprime pairs $(q, q + d)$ where $d = 2, 4$ or $6$.
The same is true for $d = 6D$ or $12D$ for any integer~$D$.}

\begin{theorem}
\label{th:8}
Let $\boldsymbol \alpha = (\alpha_1, \dots, \alpha_j)$ be any exponent pattern which includes at least one $\alpha_i = 2$ and at least three different entries equal to~$1$.
Then there exist arbitrarily long arithmetic progressions of numbers $n$ such that all $n$ and $n + 1$ in the progression have exponent pattern~$\boldsymbol\alpha$.
In particular we have arbitrarily long arithmetic progressions of integers $n$ satisfying simultaneously
\beq
\omega(n) = \omega(n + 1) = 4, \ \
\Omega(n) = \Omega(n + 1) = 5, \ \
d(n) = d(n + 1) = 24.
\label{eq:12.3}
\eeq
The above assertion is true if the triplet $(4,5,24)$ is substituted by $(4 + B,\, 5 + B,\, 24 \cdot 2^B)$ or $(5, 5+B, 24\cdot (B+1))$ for any $B \geq 0$.
The same is true (not necessarily simultaneously) for any of the equations
\beq
\omega(n) = \omega(n + 1) = A, \ \
\Omega(n) = \Omega(n + 1) = B, \ \
d(n) = d(n + 1) = C,
\label{eq:12.4}
\eeq
where $A, B, C$ are any integers with $A \geq 3$, $B \geq 4$, $24 \mid C$.
\end{theorem}

The above is a far-reaching  generalization of three conjectures of Erd\H{o}s \cite{Erd} and Erd\H{o}s--Mirsky \cite{EM}, respectively, asking whether there exists an infinite set of numbers $n$ satisfying (not necessarily simultaneously)

\hspace*{5.5pt}(i) $\omega(n) = \omega(n + 1)$  \ ($\omega(n)$ is the number of distinct prime divisors of~$n$),

\hspace*{3.5pt}(ii) $\Omega(n) = \Omega(n + 1)$  \ ($\Omega(n)$ is the total number of prime divisors of~$n$),

(iii) $d(n) = d(n + 1)$ \  ($d(n)$ is the number of divisors of~$n$).

\smallskip
We mention that  Theorems \ref{th:6}--\ref{th:9} need apart from a variant of the Green--Tao method sketched in our present Section~\ref{sec:10} also a variant of the method of \cite{GPY1}, due to S. W. Graham, D. Goldston, C. Y{\i}ld{\i}r{\i}m and the present author.
In contrast to this, in Theorem~\ref{th:5} -- which is a generalization of R\'enyi's result \cite{Ren} about the existence of infinitely many primes with $p + 2 = P_K$ for some large $K$ -- the method of \cite{GPY}, our Lemma~\ref{lem:4}, can be substituted for a result contained implicitly in Theorem~10.7 of Halberstam and Richert \cite{HR}.

Although we do not know whether any given number~$d$ appears infinitely many times as the difference of two semiprimes, we are able to prove the following rather general

\begin{theorem}
\label{th:9}
There exists an admissible $k$-tuple $\mathcal H = \{0, h_1, \dots, h_{k - 1}\}$ for any $k \geq 1$ such that there are arbitrarily long arithmetic progressions of semiprimes $q_j$ with the property that all numbers $q_j + h_i$ $(i = 1,2, \dots, k - 1)$ are semiprimes, too.
\end{theorem}

Some further results, which are not connected with the Green--Tao method, but which form a part of the proof of Theorems \ref{th:7}--\ref{th:9}, are the following

\begin{theorem}
\label{th:10}
Let $\mathcal H$ be an admissible $k$-tuple with $k \geq 3$.
The number of those $n$'s up to $N$ for which $n + \mathcal H$ contains at least two semiprimes and almost primes in all other components $n + h_s$ is
\beq
\geq c_2(k, \mathcal H) \frac{N}{\log^k N}.
\label{eq:12.5}
\eeq
\end{theorem}

\begin{theorem}
\label{th:11}
The number of integers $n$ satisfying any of the equations
\begin{alignat}2
\label{eq:12.6}
\omega(n) &= \omega(n + 1) = A, & \quad &A \text{ fixed, } A \geq 3,\\
\label{eq:12.7}
\Omega(n) &= \Omega(n + 1) = B, & \quad &B \text{ fixed, } B \geq 4,\\
\label{eq:12.8}
d(n) &= d(n + 1) = C, & \quad & C \text{ fixed, } 24\mid C
\end{alignat}
 as well as the number of integers having an exponent patter including $\{2, 1,\\ 1, 1\}$ is
 \beq
 \gg \frac{N}{\log^3 N},
 \label{eq:12.9}
 \eeq
 where the constant implied by the $\ll$ symbol depends on $A$, $B$ or $C$, respectively.
\end{theorem}

In particular, \eqref{eq:12.3} has at least $cN/\log^3 N$ solutions below~$N$.
We remark that the expected number of solutions should be $c_6 N \,(\log_2 N)^{c_7}/{\log^2 N}$.

To the proof of Theorems \ref{th:7}--\ref{th:11} we shall return in a later work.
Theorem~\ref{th:6} actually follows from the results of Sections~\ref{sec:6} and \ref{sec:10} of our present work, more precisely from Lemmas \ref{lem:1}, \ref{lem:2}, \ref{lem:4}, \ref{lem:6} and Theorem~\ref{th:5}.

\noindent
{\small J\'anos {\sc Pintz}\\
R\'enyi Mathematical Institute of the Hungarian Academy
of Sciences\\
Budapest\\
Re\'altanoda u. 13--15\\
H-1053 Hungary\\
E-mail: pintz@renyi.hu}

\end{document}